\documentclass{article}

\input{gadgets.sty}
\input{hexboard.sty}

\title{All passable games are realizable as monotone set coloring games}

\author{Eric Demer, UCLA\\
  Peter Selinger, Dalhousie University\\
  Kyle Wang, Dalhousie University}

\date{}

\begin{document}

\maketitle

\begin{abstract}
  The class of passable games was recently introduced by Selinger as a
  class of combinatorial games that are suitable for modelling
  monotone set coloring games such as Hex. In a monotone set coloring
  game, the players alternately color the cells of a board with their
  respective color, and the winner is determined by a monotone
  function of the final position. It is easy to see that every
  monotone set coloring game is a passable combinatorial game. Here we
  prove the converse: every passable game is realizable, up to
  equivalence, as a monotone set coloring game.
\end{abstract}

\section{Introduction}

A class of combinatorial games that is suitable for modelling Hex and
other monotone set coloring games was recently introduced in
{\cite{Selinger2021}}. In a monotone set coloring game
{\cite{VanRijswijck}}, the players alternately color the cells of a
board with their respective color. When all cells have been colored,
the winner is determined by a monotone function of the final board
coloring. Hex is perhaps the best-known example of a monotone set
coloring game, where the board looks as follows, and the winner is the
player (Black or White) who connects the two edges of their own color
{\cite{Hein,Nash,Gardner,Hayward-full-story}}.
\[
\begin{hexboard}
  \board(5,5)
\end{hexboard}
\]
Note that the winning condition in Hex is indeed monotone, because
the addition of black cells (and the removal of white ones) can only
help Black. Also, it is a well-known property of Hex that there are no
draws, i.e., one player necessarily wins.

Combinatorial game theory is a formalism for the study of sequential
perfect information games. It was introduced by Conway {\cite{ONAG}}
and Berlekamp, Conway, and Guy {\cite{WinningWays}}. In
{\cite{Selinger2021}}, Selinger introduced a class of combinatorial
games called ``passable'' games, and showed that the combinatorial
value of every monotone set coloring game is passable. The converse,
i.e., whether every (finite) passable game can be realized as the
value of a monotone set coloring game, was left as an open question.
In this paper, we answer this question positively.

\section{Background}

\subsection{Passable games}

We briefly recall some definitions and results that are needed in this
paper. For a more comprehensive account, see {\cite{Selinger2021}}.

\begin{definition}[Games over a poset]
  Fix a partially ordered set $A$, whose elements we call {\em
    atoms}. We define a class of combinatorial games as follows:
  \begin{itemize}
  \item For every atom $a\in A$, $[a]$ is a game.
  \item If $L$ and $R$ are non-empty sets of games, then $\g{L|R}$ is a
    game.
  \end{itemize}
  This definition is inductive, i.e., the class of games is the smallest
  class closed under the above two rules.
\end{definition}

We call a game of the form $[a]$ {\em atomic} and a game of the form
$\g{L|R}$ {\em composite}.  The elements of $L$ and $R$ are called the
game's {\em left} and {\em right options}, respectively. We use the
standard conventions of combinatorial game theory; for example, we
write $G^L$ and $G^R$ for a typical left and right option of $G$. Note
that atomic games do not have options, so if $G=[a]$ is an atomic
game, any statement of the form ``for all $G^L$'' is vacuously true,
and any statement of the form ``there exists $G^L$'' is trivially
false. We sometimes write $a$ instead of $[a]$ for an atomic game when
no confusion arises. The {\em positions} of a game $G$ are $G$ itself,
all of its left and right options, their left and right options, and
so on recursively.  A game $G$ is called {\em finite} or {\em short}
if $G$ has finitely many positions. In this paper, we assume that all
atom posets have a top element, which we denote $\top$, and a bottom
element, which we denote $\bot$.

\begin{definition}[Order]
  On the class of games, we define two relations $\leq$ and $\tri$ by
  mutual recursion as follows.
    \begin{itemize}
  \item $G\leq H$ if all three of the following conditions hold:
    \begin{enumerate}
    \item All left options $G^L$ satisfy $G^L\tri H$, and
    \item all right options $H^R$ satisfy $G\tri H^R$, and
    \item if $G$ or $H$ is atomic, then $G\tri H$.
    \end{enumerate}
  \item $G\tri H$ if at least one of the following
    conditions holds:
    \begin{enumerate}
    \item There exists a right option $G^R$ such that $G^R\leq H$, or
    \item there exists a left option $H^L$ such that $G\leq H^L$, or
    \item $G=[a]$ and $H=[b]$ are atomic and $a\leq b$.
    \end{enumerate}
  \end{itemize}
\end{definition}

Note that when the games are composite, this definition coincides with
the standard definition of $\leq$ and $\tri$ in combinatorial game
theory {\cite{ONAG}}. The only difference is in the treatment of
atomic games. See {\cite{Selinger2021}} for a detailed explanation of
why this definition makes sense.

We say that games $G$ and $H$ are {\em equivalent}, in symbols $G\eq H$, if
$G\leq H$ and $H\leq G$. We sometimes refer to an equivalence class of
games as a {\em value}, i.e., we say that two games have the same
value if and only if they are equivalent. Just like in standard
combinatorial game theory, there is a notion of {\em canonical form}
of our games, such that each game is equivalent to a unique canonical
form. See {\cite{Selinger2021}} for details.

As usual in combinatorial game theory, we say that $H$ is a {\em left
  gift horse} for $G$ if $H\tri G$. The {\em gift horse lemma} states
that for a composite game $G$, we have $G\eq\g{H,G^L|G^R}$ if and only
if $H\tri G$. Of course the dual of this statement is also true, i.e.,
$H$ is a {\em right gift horse} for $G$ if $G\tri H$, and in this
case, $G\eq\g{G^L|G^R,H}$.

\begin{definition}[Monotone and passable games]
  If $G$ is a composite game, we say that a left option $G^L$ is {\em
    good} if $G\leq G^L$. Dually, we say that a right option $G^R$ is
  {\em good} if $G^R\leq G$. A game is called {\em locally monotone}
  if all of its left and right options are good, and {\em locally
    passable} if it is atomic or has at least one good left option or
  one good right option. A game $G$ is called {\em monotone}
  (respectively {\em passable}) if all positions of $G$ are locally
  monotone (respectively locally passable).
\end{definition}

Intuitively, a good option is one that helps the player who plays it.
In a monotone game, all options are good, so making any move is always
at least as good as passing. In a passable game, there is at least one
good option, so even if passing were allowed, no player would benefit
from doing so. This is explained in more detail in
{\cite{Selinger2021}}. We note that $G$ is locally passable if and
only if $G\tri G$. Trivially, every monotone game is passable. A kind
of converse is given by the following theorem:

\begin{theorem}[Fundamental theorem of passable games {\cite{Selinger2021}}]
  \label{thm:fundamental}
  Every passable game is equivalent to a monotone game.
\end{theorem}

We also need a few operations on games.

\begin{definition}[Sum]
  Let $G$ and $H$ be games over atom posets $A$ and $B$, respectively.
  Then their {\em sum} $G+H$ is a game over the cartesian product
  poset $A\times B$, and is defined recursively as follows:
  \begin{itemize}
  \item If at least one of $G$ or $H$ is composite:
    \[ G+H = \g{G^L+H, G+H^L | G^R+H, G+H^R}.
    \]
  \item If $G=[a]$ and $H=[b]$ are atomic:
    \[ G+H = [(a,b)].
    \]
  \end{itemize}
\end{definition}

Note that by our convention, the definition simplifies in case one of
$G,H$ is atomic and the other is not. For example, if $G$ is atomic
and $H$ is not, then $G$ has no left or right options, so $G+H=
\g{G+H^L | G+H^R}$. As explained in {\cite{Selinger2021}}, sums do not
in general respect equivalence of games; however, when the games are
passable, they do. Specifically, if $G,G',H,H'$ are passable and $G\eq
G'$ and $H\eq H'$, then $G+H$ and $G'+H'$ are passable and $G+H\eq
G'+H'$.

\begin{definition}[Map operation]
  Let $G$ be a game over $A$, and let $f:A\to B$ be a monotone
  function between atom posets. Then $f(G)$ is a game over $B$,
  defined by applying $f$ to all atomic positions in $G$. More
  formally, $f(G)$ is defined recursively as follows: $f([a]) =
  [f(a)]$ and $f(\g{G^L|G^R}) = \g{f(G^L)|f(G^R)}$.
\end{definition}

It is often convenient to combine the sum operation with the map
operation. For example, if $G,H,K$ are games over posets $A,B,C$ and
$f:A\times B\times C\to D$ is a monotone function, then $f(G+H+K)$ is
a game over $D$.

If $S$ and $T$ are sets of games, we say that $S$ and $T$ are {\em
  left equivalent} if for all $G^L,G^R$, we have
$\g{S,G^L|G^R}\eq\g{T,G^L|G^R}$. Right equivalence is defined dually.
We define $\upl S = \g{\top|\g{S|\bot}}$. Then $\upl S$ is
left equivalent to $S$. Dually, $\downr S = \g{\g{\bot|S}|\top}$ is
right equivalent to $S$.

\subsection{Monotone set coloring games}

Fix a partially ordered set $A$ of atoms. Let $\Bool=\s{\top,\bot}$ be
the set of booleans. Here, $\bot$ denotes ``false'' or ``bottom'', and
$\top$ denotes ``true'' or ``top'', with the natural order
$\bot < \top$. If $X$ is any set, we write $\Bool^X$ for the set of
functions from $X$ to $\Bool$. These functions are equipped with the
pointwise order, i.e., $f\leq g$ if and only if for all $x\in X$,
$f(x)\leq g(x)$.

\begin{definition}
  A {\em monotone set coloring game} over $A$ is a pair $S=(|S|,\payoff{S})$,
  where $|S|$ is a finite set and $\payoff{S}:\Bool^{|S|}\to A$ is a monotone
  function.  The set $|S|$ is called the {\em carrier} of the game,
  and its elements are called {\em cells}. The function $\payoff{S}$ is
  called the {\em payoff function}. We sometimes write $S:A$ to
  indicate that $S$ is a monotone set coloring game over $A$.
\end{definition}

If $S:A$ is a monotone set coloring game, a {\em position} is a map $p
: |S| \to \s{\top,\bot,\star}$. Here, $\bot$ indicates a cell occupied
by the right player, $\top$ indicates a cell occupied by the left
player, and $\star$ indicates an empty cell. When $p(c)=\top$, we also
say that the cell $c$ is {\em colored} with color $\top$, and
similarly when $p(c)=\bot$. We write $\emptypos_S$ for the all-empty
position. A position is {\em atomic} if there are no empty cells.

To play the game, the players take turns coloring a cell in their own
color, until an atomic position is reached, at which point the payoff
function is used to assign a value to the final position. The left
player's goal is to maximize the payoff, and the right player's goal
is to minimize it.

If $p$ is a non-atomic position, we write $p^L$ for a typical left
option, i.e., a position obtained from $p$ by coloring exactly one
empty cell with $\top$. Similarly, $p^R$ is a typical right option,
i.e., a position obtained from $p$ by coloring exactly one empty cell
with $\bot$. Then a monotone set coloring game can be regarded as a
combinatorial game in an obvious way, which is made precise in the
following definition.

\begin{definition}
  Let $S:A$ be a monotone set coloring game. To each position $p$ of
  $S$, we associate a combinatorial game $\sem{p}$ over $A$, defined
  recursively as follows: $\sem{p} = [f(p)]$ if $p$ is atomic, and
  $\sem{p} = \g{\sem{p^L}|\sem{p^R}}$ if $p$ is non-atomic. We
  identify the game $S$ itself with its initial position, i.e., we
  define $\sem{S} = \sem{\emptypos_S}$.
\end{definition}

\begin{example}\label{exa:notation}
  In this paper, we will sometimes describe a monotone set coloring
  game by a notation such as the following.
  \[
  a: \s{\ciw\ciw\ciw\cib\cib,\cib\ciw\cib\ciw\cib,\cib\cib\ciw\ciw\ciw},\quad
  b: \s{\ciw\ciw\cib\ciw\ciw,\ciw\cib\ciw\cib\ciw},
  \]
  This notation is interpreted as follows. The game has 5 cells, and
  is over the poset $P_4=\s{\top,a,b,\bot}$, where $\top$ is the top
  element, $\bot$ is the bottom element, and $a,b$ are incomparable.
  
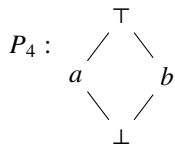
\[
  \begin{tikzpicture}[scale=0.4]
    \node at (-3,1) {$P_4:$};
    \node(top) at (0,2) {$\top$};
    \node(a) at (-1.5,0) {$a$};
    \node(b) at (1.5,0) {$b$};
    \node(bot) at (0,-2) {$\bot$};
    \draw (bot) -- (a) -- (top);
    \draw (bot) -- (b) -- (top);
  \end{tikzpicture}
  \]
  We call the left player ``Black'' and the right player ``White'',
  and we write $\cib$ and $\ciw$ to denote a cell occupied by Black
  and White, respectively.  If the final position $p$ satisfies
  $p\geq\ciw\ciw\ciw\cib\cib$ or $p\geq\cib\ciw\cib\ciw\cib$ or
  $p\geq\cib\cib\ciw\ciw\ciw$, then Black achieves at least outcome
  $a$. If $p\geq\ciw\ciw\cib\ciw\ciw$ or $p\geq\ciw\cib\ciw\cib\ciw$,
  then Black achieves at least outcome $b$.  If both these conditions
  hold, then the outcome is $\top$, and if neither holds, the outcome
  is $\bot$. For example, if the final position is
  $p=\ciw\cib\ciw\cib\cib$, the outcome is $\top$ because
  $p\geq \ciw\ciw\ciw\cib\cib$ and $p\geq\ciw\ciw\cib\ciw\ciw$.  If
  the final position is $p=\cib\cib\ciw\ciw\cib$, the outcome is $a$
  because $p\geq\cib\cib\ciw\ciw\ciw$, but $p$ is not above
  $p\geq\ciw\ciw\cib\ciw\ciw$ or $p\geq\ciw\cib\ciw\cib\ciw$. If the
  final position is $p=\cib\ciw\ciw\ciw\cib$, then the outcome is
  $\bot$.

  The combinatorial value of the above game is
  $\g{a,\g{\top|b}|\g{a|\bot},b}$, as can be checked by computation.
\end{example}

\begin{example}
  The game of Hex is a monotone set coloring game. For example, Hex of
  size $2\times 2$ can be described as the monotone set coloring game
  $S:\Bool$ where $|S|=\s{a1, a2, b1, b2}$ and
  \[
  \payoff{S}(p) = \begin{choices}
    \top & \mbox{if $p(a1)=\top$ and $p(a2)=\top$,} \\
    \top & \mbox{if $p(b1)=\top$ and $p(a2)=\top$,} \\
    \top & \mbox{if $p(b1)=\top$ and $p(b2)=\top$,} \\
    \bot & \mbox{otherwise.}
  \end{choices}
  \]
  Here again, we identify the color $\top$ with black and $\bot$ with
  white. The winning condition is exactly the usual one, i.e., Black
  wins if and only if the black edges are connected by black stones.
  
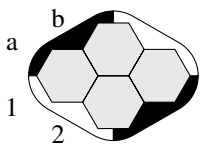
\[
  \begin{hexboard}
    \board(2,2)
    \cell(1,-0.2)\label{a}
    \cell(2,-0.2)\label{b}
    \cell(-0.2,1)\label{1}
    \cell(-0.2,2)\label{2}
  \end{hexboard}
  \]
  Or in the notation of Example~\ref{exa:notation}, this game can be
  concisely described as
  \[
  \top : \s{\cib\cib\ciw\ciw, \ciw\cib\cib\ciw, \ciw\ciw\cib\cib}.
  \]
  The combinatorial value of this game is $\g{\top|\bot}$, i.e., it is
  a first-player win.
\end{example}

\begin{remark}\label{rem:monotone-set-is-monotone}
  If $S$ is a monotone set coloring game, $\sem{S}$ is always a
  monotone game. This is clear because in a monotone set coloring
  game, having an extra cell is never bad for a player.
  In particular, it follows that $\sem{S}$ is passable.
\end{remark}

\section{The main result}

In this section, we will prove our main result, namely, that every
finite passable game is realizable as a monotone set coloring game.

\begin{definition}
  Let $G$ be a combinatorial game over an atom poset $A$. Then $G$ is
  {\em realizable as a monotone set coloring game}, or briefly {\em
    realizable}, if there exists a monotone set coloring game $S:A$
  such that $G\eq \sem{S}$.
\end{definition}

Note that by Remark~\ref{rem:monotone-set-is-monotone}, every
realizable game is equivalent to a passable one. Our main result is
the following theorem, which states the converse.

\begin{theorem}[Realizability theorem for monotone set coloring games]\label{thm:main}
  All finite passable games are realizable as monotone set coloring
  games.
\end{theorem}

The rest of this section is devoted to the proof of this theorem. The
proof proceeds in three steps: in Section~\ref{ssec:new-gadgets}, we
show that the class of finite passable games can be characterized, up
to equivalence, as the smallest class of games containing the atomic
games and closed under four operations, which we call gadgets. In
Section~\ref{ssec:new-closure}, we show that any class of games that
contains a small number of specific simple games and is closed under
the sum and map operations is automatically closed under the four
gadgets. In Section~\ref{ssec:new-proof}, we show that the class of
games that are realizable as monotone set coloring games satisfies the
above conditions, and therefore contains all passable games up to
equivalence. This proves the theorem.

We note that Sections~\ref{ssec:new-gadgets} and
{\ref{ssec:new-closure}} are not about monotone set coloring games;
they contain results about passable games that are of independent
interest. Only the results of Section~\ref{ssec:new-proof} are
specifically about monotone set coloring games.

\subsection{A characterization of finite passable games}
\label{ssec:new-gadgets}

Fix an atom poset $A$. By a {\em gadget}, we mean an operation that
produces a game over $A$ from one or more smaller such games. The
following proposition characterizes the class of finite passable games
as being the smallest class that contains the atomic games and is
closed under four specific gadgets.

\begin{proposition}\label{prop:gadgets}
  Up to equivalence of games, the class of finite passable games over
  $A$ is the smallest class $\Cc$ of games containing the atomic games
  and closed under the following four gadgets:
  \begin{enumerate}\alphalabels
  \item Left forcing gadget: if $G\in \Cc$, then $\g{\top|G}\in \Cc$.
    \label{prop:gadgets-a1}
  \item Right forcing gadget: if $G\in \Cc$, then $\g{G|\bot}\in \Cc$.
    \label{prop:gadgets-a2}
  \item Choice gadget: if $G,H\in\Cc$, then
    $\g{\g{\top|G},\g{\top|H}|\g{G|\bot},\g{H|\bot}}\in\Cc$.
    \label{prop:gadgets-b}
  \item Coupling gadget: if $G,H\in\Cc$, then
    $\g{G,\g{\top|H}|\g{G|\bot},H}\in\Cc$.
    \label{prop:gadgets-c}
  \end{enumerate}
\end{proposition}

We collectively call the four gadgets of
Proposition~\ref{prop:gadgets} the {\em primitive gadgets}. By virtue
of the proposition, they can be viewed as the basic building blocks of
all finite passable games.  Before we prove the proposition, we
briefly comment on the meaning of the primitive gadgets. The forcing
gadgets force one of the players to immediately make a move. The
choice gadget allows the player whose turn it is to choose between
playing $G$ and $H$, where it will then be that player's turn
again. The coupling gadget is superficially similar to the choice
gadget, in that it, too, allows the player whose turn it is to choose
between $G$ and $H$. However, whose turn it will be following this
choice does not depend on which player made the choice: it is always
the right player's turn if $G$ is chosen and the left player's turn if
$H$ is chosen.

We first prove the ``easy'' direction of
Proposition~\ref{prop:gadgets}, namely, that the class of finite passable
games has the required closure properties.

\begin{lemma}\label{lem:passable-closure}
  All atomic games are passable, and the class of finite passable
  games is closed under the four gadgets of
  Proposition~\ref{prop:gadgets}.
\end{lemma}

\begin{proof}
  Atomic games are passable by definition. Now suppose $G$ and $H$ are
  passable. Then $\g{\top|G}$ is clearly passable, since all its
  options are passable and $\top$ is a good left option. The case of
  $\g{G|\bot}$ is dual. To show that
  $K=\g{\g{\top|G},\g{\top|H}|\g{G|\bot},\g{H|\bot}}$ is passable,
  note that we just proved that all of its options are passable, so it
  suffices to show that $K$ has a good left option. Since $G$ is
  passable, we have $G\tri G$, hence $\g{G|\bot}\leq G$, hence
  $K\leq\g{\top|G}$. Hence $\g{\top|G}$ is a good left option of
  $K$. To show that $J=\g{G,\g{\top|H}|\g{G|\bot},H}$ is passable, it
  likewise suffices to show that there exists a good left option;
  clearly $\g{\top|H}$ is such an option.
\end{proof}

Note that Lemma~\ref{lem:passable-closure} holds ``on the nose'',
i.e., not just up to equivalence.

To finish the proof of Proposition~\ref{prop:gadgets}, we must show
that if a class of games has the required closure properties, then it
contains all finite passable games up to equivalence. Equivalently, we
must show that every finite passable game can be built, up to
equivalence, from atomic games by means of the four gadgets. To
facilitate the proof, we first show that several other useful gadgets
can be constructed from the primitive ones. Throughout the following
lemmas, let $\Cc$ be a class of games that is closed under the four
primitive gadgets of Proposition~\ref{prop:gadgets} and under
equivalence.

\begin{lemma}[One-sided binary choice gadget]\label{lem:one-sided-choice}
  If $G,H\in\Cc$, then $\g{G,H|\bot}\in\Cc$.
\end{lemma}

\begin{proof}
  Assume $G,H\in\Cc$. By the right forcing gadget, we have
  $\g{G|\bot}\in\Cc$ and $\g{H|\bot}\in\Cc$. Applying the choice
  gadget to $\g{G|\bot}$ and $\g{H|\bot}$, we get
  \[
  K = \g{\g{\top|\g{G|\bot}},\g{\top|\g{H|\bot}}|\g{\g{G|\bot}|\bot},\g{\g{H|\bot}|\bot}}\in\Cc.
  \]
  Note that $\g{\top|\g{G|\bot}} = \upl G$ is left equivalent to $G$,
  and $\g{\g{G|\bot}|\bot}$ is equivalent to $\bot$, and similarly for
  $H$. Therefore, $K \eq \g{G,H|\bot}$, and since $\Cc$ is closed
  under equivalence, $\g{G,H|\bot}\in\Cc$ as claimed.
\end{proof}

\begin{lemma}[One-sided $n$-ary choice gadget]\label{lem:nary-one-sided-choice}
  If $G_1,\ldots,G_n\in\Cc$, then $\g{G_1,\ldots,G_n|\bot}\in\Cc$.
\end{lemma}

\begin{proof}
  By induction on $n\geq 1$. The base cases are the right forcing
  gadget for $n=1$ and the one-sided binary choice gadget of
  Lemma~\ref{lem:one-sided-choice} for $n=2$. Now consider $n\geq 3$. By
  the induction hypothesis,
  $\g{G_1,\ldots,G_{n-1}|\bot}\in\Cc$. Therefore, by the left forcing
  gadget, $\g{\top|\g{G_1,\ldots,G_{n-1}|\bot}} =
  \upl(G_1,\ldots,G_{n-1})\in\Cc$. By Lemma~\ref{lem:one-sided-choice},
  $\g{\upl(G_1,\ldots,G_{n-1}),G_n|\bot}\in\Cc$. Since
  $\upl(G_1,\ldots,G_{n-1})$ is left equivalent to
  $G_1,\ldots,G_{n-1}$, it follows that
  $\g{G_1,\ldots,G_n|\bot}\in\Cc$ as claimed.
\end{proof}

We need two more lemmas before we can prove Proposition~\ref{prop:gadgets}.

\begin{lemma}[Coupling lemma]\label{lem:coupling}
  If $G, H\in\Cc$ and $K = \g{G|H}$ is locally monotone, then
  $K\in\Cc$.
\end{lemma}

\begin{proof}
  By definition of local monotonicity, we have $H\leq K \leq G$.  It
  follows that $\g{\top|H}\tri K$ and $K\tri \g{G|\bot}$. Therefore
  $\g{\top|H}$ and $\g{G|\bot}$ are left and right gift horses for
  $K$, respectively. Thus $K$ is equivalent to
  $\g{G,\g{\top|H}|\g{G|\bot},H}$, which is in $\Cc$ by the coupling
  gadget. Since $\Cc$ is closed under equivalence, we have $K\in\Cc$.
\end{proof}

The following is a generalization of the coupling lemma to games with
more than two options. A game is called {\em locally semi-monotone} if
it has at least one good left option and at least one good right
option. Every monotone game is semi-monotone, and every
semi-monotone game is passable.

\begin{lemma}[$(n,m)$-ary coupling lemma]\label{lem:nary-coupling}
  If $G_1,\ldots,G_n,H_1,\ldots,H_m\in\Cc$ and
  $K=\g{G_1,\ldots,G_n|H_1,\ldots,H_m}$ is locally semi-monotone, then
  $K\in\Cc$.
\end{lemma}

\begin{proof}
  Let $G=\g{G_1,\ldots,G_n|\bot}$ and
  $H=\g{\top|H_1,\ldots,H_m}$. Then $G,H\in\Cc$ by
  Lemma~\ref{lem:nary-one-sided-choice} and its dual. By the forcing
  gadgets, we also have $\g{\top|G}\in\Cc$ and $\g{H|\bot}\in\Cc$.
  Let $K'=\g{\g{\top|G}|\g{H|\bot}}$. Note that the left option of
  $K'$ is $\g{\top|G}=\upl(G_1,\ldots,G_n)$, which is left equivalent
  to $G_1,\ldots,G_n$, and similarly for the right option. Hence
  $K'\eq K$.  It is easy to see that $K'$ is locally monotone; namely,
  since $K$ is semi-monotone, it has some good left option $G_i$. Then
  we have $K\leq G_i$, and therefore $K\tri G$, which implies
  $K\leq\g{\top|G}$, hence $K'\leq\g{\top|G}$. The proof of
  $\g{H|\bot}\leq K'$ is dual. Therefore, the hypotheses of
  Lemma~\ref{lem:coupling} are satisfied. Hence $K'\in\Cc$, and
  therefore $K\in\Cc$.
\end{proof}

\begin{proof}[Proof of Proposition~\ref{prop:gadgets}]
  By Lemma~\ref{lem:passable-closure}, the class of finite passable
  games over $A$ satisfies the closure properties of
  Proposition~\ref{prop:gadgets}. What we need to show is that it is
  the smallest such class up to equivalence. So let $\Cc$ be any class
  of games containing all atomic games and closed under the four
  primitive gadgets and under equivalence of games. We must show that
  all finite passable games belong to $\Cc$. Since by the fundamental
  theorem of passable games, every passable game is equivalent to a
  monotone game, it is sufficient to show that all finite monotone
  games belong to $\Cc$. We will show this by induction on games. For
  the base case, note that all atomic games belong to $\Cc$ by
  assumption. For the induction step, consider some monotone game
  $K=\g{G_1,\ldots,G_n|H_1,\ldots,H_m}$. By the induction hypothesis,
  $G_1,\ldots,G_n,H_1,\ldots,H_m\in\Cc$. Then $K\in\Cc$ by
  Lemma~\ref{lem:nary-coupling}.
\end{proof}

\begin{remark}
  The coupling gadget $\g{G,\g{\top|H}|\g{G|\bot},H}$ is self-dual.
  In Proposition~\ref{prop:gadgets}, we could have equivalently
  replaced it by the following {\em left-biased coupling gadget}:
  $\g{\g{\top|G'},\g{\top|H}|G',H}$.  Indeed, in the presence of
  forcing, the two are easily seen to be equivalent via the
  substitutions $G'=\g{G|\bot}$ and $G=\g{\top|G'}$. Thus, while it is
  an essential feature of the coupling gadget that the player whose
  turn it is in $G$ or $H$ does not depend on which player makes the
  choice of whether to play $G$ or $H$, it is inessential who that
  player is. Of course there is also a dual {\em right-biased coupling
    gadget}.
\end{remark}

We end this subsection by pointing out that
Proposition~\ref{prop:gadgets} gives rise to an induction principle
for passable games. To show that all passable games have some
property, it suffices to show that the property is invariant under
equivalence, holds for atomic games, and is closed under the four
primitive gadgets. This is usually more convenient than working
directly by induction on the definition of passable games, because the
latter requires each induction step to use the property of local
passability, which is not itself an inductive property.

\subsection{Making gadgets from games}\label{ssec:new-closure}

In this section, we show that when a class of games is closed under
the sum and map operations, then the existence of four specific
games is sufficient to guarantee closure under the primitive gadgets
of Proposition~\ref{prop:gadgets}.

Let $P_3=\s{\top,a,\bot}$ be the 3-element linearly ordered poset, and
let $P_4=\s{\top,a,b,\bot}$ be the 4-element poset from
Example~\ref{exa:notation}.  Let $A$ be any poset with top and bottom
elements. We define functions $f:P_3\times A \to A$ and $g:P_4\times
A\times A\to A$ by
\[
f(x,y) =
\begin{choices}
  \top & \mbox{if $x=\top$,} \\
  y & \mbox{if $x=a$,} \\
  \bot & \mbox{if $x=\bot$}
\end{choices}
\qquad\mbox{and}\qquad
g(x,y,z) =
\begin{choices}
  \top & \mbox{if $x=\top$,} \\
  y & \mbox{if $x=a$,} \\
  z & \mbox{if $x=b$,} \\
  \bot & \mbox{if $x=\bot$.}
\end{choices}
\]
Then $f$ and $g$ are monotone functions. Therefore, if $X$
is a game over $P_3$ and $G$ is a game over $A$, then the sum $f(X+G)$
is a well-defined game over $A$. Similarly, if $X$ is a game over
$P_4$ and $G,H$ are games over $A$, the sum $g(X+G+H)$ is a
well-defined game over $A$. For brevity, we will henceforth write
these sums simply as $X+G$ or $X+G+H$, i.e., the functions $f$ and $g$
will be understood from the context. 

\begin{lemma}\label{lem:gadgets}
  Let $G,H$ be games over $A$. Then the following hold:
  \begin{enumerate}\alphalabels
  \item $\g{\top|a}+G  ~\eq~ \g{\top|G}$.
    \label{lem:gadgets-a1}
  \item $\g{a|\bot}+G  ~\eq~ \g{G|\bot}$.
    \label{lem:gadgets-a2}
  \item $\g{\g{\top|a},\g{\top|b}|\g{a|\bot},\g{b|\bot}}+G+H ~\eq~ \g{\g{\top|G},\g{\top|H}|\g{G|\bot},\g{H|\bot}}$.
    \label{lem:gadgets-b}
  \item $\g{a,\g{\top|b}|\g{a|\bot},b}+G+H ~\eq~ \g{G,\g{\top|H}|\g{G|\bot},H}$.
    \label{lem:gadgets-c}
  \end{enumerate}
\end{lemma}

\begin{proof}
  \begin{enumerate}\alphalabels
  \item By induction on $G$. Let $X=\g{\top|a}$ and $K=\g{\top|G}$. We
    want to show $X+G\eq K$. The left options of $X+G$ are:
    \begin{itemize}
    \item $\top+G$. This is easily seen to be equivalent to $\top$.
    \item $X+G^L$. This is equivalent to $\g{\top|G^L}$ by the
      induction hypothesis.
    \end{itemize}
    The right options of $\g{\top|a}+G$ are:
    \begin{itemize}
    \item $a+G$. This is easily seen to be equivalent to $G$.
    \item $X+G^R$. This is equivalent to $\g{\top|G^R}$ by the
      induction hypothesis.
    \end{itemize}
    Therefore, $X+G\eq \g{\top,\g{\top|G^L}|G,\g{\top|G^R}}$.  Note
    that $\g{\top|G^L}\tri \g{\top|G}$ and $\g{\top|G}\tri
    \g{\top|G^R}$. Therefore, $\g{\top|G^L}$ is a left gift horse and
    $\g{\top|G^R}$ is a right gift horse for $K$. By the gift horse
    lemma, it follows that $X+G\eq K$ as claimed.
  \item This is the dual of {\eqref{lem:gadgets-a1}}.
  \item By induction on $G$ and $H$. Let
    $X=\g{\g{\top|a},\g{\top|b}|\g{a|\bot},\g{b|\bot}}$ and
    $K=\g{\g{\top|G},\g{\top|H}|\g{G|\bot},\g{H|\bot}}$. We want to
    show $X+G+H\eq K$. The left options of $X+G+H$ are:
    \begin{itemize}
    \item $\g{\top|a}+G+H$. This is equivalent to $\g{\top|G}$ by the
      same argument as in {\eqref{lem:gadgets-a1}}.
    \item $\g{\top|b}+G+H$. This is equivalent to $\g{\top|H}$ by the
      same argument as in {\eqref{lem:gadgets-a1}}.
    \item $X+G^L+H$. This is equivalent to
      $\g{\g{\top|G^L},\g{\top|H}|\g{G^L|\bot},\g{H|\bot}}$ by the
      induction hypothesis.
    \item $X+G+H^L$. This is equivalent to
      $\g{\g{\top|G},\g{\top|H^L}|\g{G|\bot},\g{H^L|\bot}}$ by the
      induction hypothesis.
    \end{itemize}
    The right options are dual. We
    have $G^L\tri G$, hence $G^L\leq\g{\top|G}$,
    hence $G^L\tri K$,
    hence $\g{G^L|\bot}\leq K$,
    hence $\g{\g{\top|G^L},\g{\top|H}|\g{G^L|\bot},\g{H|\bot}}\tri
    K$. Therefore $X+G^L+H$ is a left gift horse for $K$.  Similarly,
    $X+G+H^L$ is also a left gift horse for $K$. The situation for the
    right options is dual. By the gift horse lemma, it follows that
    $X+G+H \eq K$ as claimed.
    \item By induction on $G$ and $H$. Let
      $X=\g{a,\g{\top|b}|\g{a|\bot},b}$ and
      $K=\g{G,\g{\top|H}|\g{G|\bot},H}$. We want to show $X+G+H\eq
      K$. The left options of $X+G+H$ are:
    \begin{itemize}
    \item $a+G+H$. This is easily seen to be equivalent to $G$.
    \item $\g{\top|b}+G+H$. This is equivalent to $\g{\top|H}$ by the
      same argument as in {\eqref{lem:gadgets-a1}}.
    \item $X+G^L+H$. This is equivalent to $\g{G^L,\g{\top|H}|\g{G^L|\bot},H}$ 
        by the induction hypothesis.
    \item $X+G+H^L$. This is equivalent to $\g{G,\g{\top|H^L}|\g{G|\bot},H^L}$ 
        by the induction hypothesis.
    \end{itemize}
    The right options are dual.  Since
    $\g{G^L,\g{\top|H}|\g{G^L|\bot},H}\leq\g{\top|H}$, we have
    $\g{G^L,\g{\top|H}|\g{G^L|\bot},H}\tri K$. Therefore, $X+G^L+H$ is
    a left gift horse for $K$.  Similarly, since $G$ is a left option
    of $K$, we have $\g{G|\bot}\leq K$, hence
    $\g{G,\g{\top|H^L}|\g{G|\bot},H^L}\tri K$. Therefore $X+G+H^L$ is
    also a left gift horse for $K$.  The situation for the right
    options is dual. By the gift horse lemma, it follows that
    $X+G+H\eq K$ as claimed.\qedhere
  \end{enumerate}
\end{proof}

\begin{remark}
  Although it is tempting to believe that the analogue of
  Lemma~\ref{lem:gadgets} holds for all games over $P_3$ and/or $P_4$,
  this is not the case. For example, it is not in general the case
  that $\g{\g{\top|a}|a} + G \eq \g{\g{\top|G}|G}$. In fact, we have
  $\g{\g{\top|a}|a} + G \eq G$. We say that a game $X$ over $P_3$ is a
  {\em gadget game} if $X+G$ is equivalent to the game obtained from
  $X$ by replacing all occurrences of the atom $a$ by $G$, and
  similarly for games over $P_4$. An analogous definition can also be
  given for $P_n$ with $n>4$. In particular, Lemma~\ref{lem:gadgets}
  states that the games $\g{\top|a}$,~ $\g{a|\bot}$,~
  $\g{\g{\top|a},\g{\top|b}|\g{a|\bot},\g{b|\bot}}$, and
  $\g{a,\g{\top|b}|\g{a|\bot},b}$ are gadget games. There are many
  other gadget games too, but for the results of this paper, we only
  need the above four.
\end{remark}

\begin{corollary}\label{cor:closure}
  Let $\Cc$ be a class of games over posets with top and bottom
  elements. If $\Cc$ is closed under equivalence, sums, and the map
  operation of monotone functions (or at least of the functions $f$
  and $g$ above), and if $\Cc$ contains the four gadget games
  $\g{\top|a}$,~ $\g{a|\bot}$,~
  $\g{\g{\top|a},\g{\top|b}|\g{a|\bot},\g{b|\bot}}$, and
  $\g{a,\g{\top|b}|\g{a|\bot},b}$, then $\Cc$ is closed under the
  primitive gadgets of Proposition~\ref{prop:gadgets}.
\end{corollary}

\subsection{Realizability by monotone set coloring games}
\label{ssec:new-proof}

We want to show that all finite passable games are realizable as
monotone set coloring games. Due to Proposition~\ref{prop:gadgets} and
Corollary~\ref{cor:closure}, it suffices to show that the class of
realizable games contains all atomic games, is closed under the sum
and map operations, and contains the four primitive gadget games.
We prove each of these properties in turn.

\begin{lemma}\label{lem:atom-realizable}
  Let $A$ be a poset and $a\in A$. The atomic game $[a]$ is realizable.
\end{lemma}

\begin{proof}
  The game $[a]$ is trivially realizable as a monotone set coloring
  game with empty carrier and constant payoff function.
\end{proof}

\begin{lemma}\label{lem:sum-realizable}
  Let $A$ and $B$ be posets, and let $G$ and $H$ be games over $A$ and
  $B$, respectively. If $G$ and $H$ are realizable, then so is $G+H$.  
\end{lemma}

\begin{proof}
  Let $S:A$ and $T:B$ be monotone set coloring games realizing $G$ and
  $H$, respectively. We define a new game $S+T$ as follows. Its
  carrier is the disjoint union of $|S|$ and $|T|$. The payoff
  function is given by
  \[
  \payoff{S+T}(p) = (\payoff{S}(p\restr{S}), \payoff{T}(p\restr{T})).
  \]
  An easy induction shows that $\sem{S+T} = \sem{S} + \sem{T}$.
\end{proof}

\begin{lemma}\label{lem:map-realizable}
  Let $A$ and $B$ be posets, let $G$ be a game over $A$, and let
  $f:A\to B$ be a monotone function. If $G$ is realizable, then so is
  $f(G)$.
\end{lemma}

\begin{proof}
  Let $S:A$ be a monotone set coloring game realizing $G$. We define a
  new game $f(S)$ with the same carrier as $S$, and payoff function
  $\payoff{f(S)}(p) = f(\payoff{S}(p))$. An easy induction shows that
  $\sem{f(S)}=f(\sem{S})$.
\end{proof}

\begin{lemma}\label{lem:realizable-values}
  The following values are realizable as monotone set
  coloring games (over $P_3$ or $P_4$, as appropriate):
  \begin{enumerate}\alphalabels
  \item $\g{\top|a}$.
    \label{lem:realizable-values-a1}
  \item $\g{a|\bot}$.
    \label{lem:realizable-values-a2}
  \item $\g{\g{\top|a},\g{\top|b}|\g{a|\bot},\g{b|\bot}}$.
    \label{lem:realizable-values-b}
  \item $\g{a,\g{\top|b}|\g{a|\bot},b}$.
    \label{lem:realizable-values-c}
  \end{enumerate}
\end{lemma}

\begin{proof}
  One can verify by direct calculation that the following games
  realize the claimed values. Note that
  {\eqref{lem:realizable-values-a1}} and
  {\eqref{lem:realizable-values-a2}} are duals, and we already
  encountered the game {\eqref{lem:realizable-values-c}} in
  Example~\ref{exa:notation}.
  \begin{enumerate}\alphalabels
  \item $a: \s{\ciw}$, $\top: \s{\cib}$.
  \item $a: \s{\cib}$, $\top: \emptyset$.
  \item $a: \s{\ciw\cib}$, $b: \s{\cib\ciw}$.
  \item $a:
    \s{\ciw\ciw\ciw\cib\cib,\cib\ciw\cib\ciw\cib,\cib\cib\ciw\ciw\ciw}$,
    $b: \s{\ciw\ciw\cib\ciw\ciw,\ciw\cib\ciw\cib\ciw}$.
    \qedhere
  \end{enumerate}
\end{proof}

Appendix~\ref{app:a} contains a list of all values over $P_4$ that are
realizable as monotone set coloring games with up to 5 cells.

\begin{proof}[Proof of Theorem~\ref{thm:main}]
  By Lemmas~\ref{lem:atom-realizable}--\ref{lem:realizable-values},
  the class of realizable games contains all atomic games, is closed
  under the sum and map operations, and contains the four primitive
  gadget games. By definition, it is also closed under equivalence. By
  Corollary~\ref{cor:closure}, this class of games is closed under the
  primitive gadgets, and therefore by Proposition~\ref{prop:gadgets},
  it contains all finite passable games.
\end{proof}

\section{The size of set coloring gadgets}

Theorem~\ref{thm:main} was only concerned with the existence of
monotone set coloring games, and not with their size. We can define
the {\em size} of a monotone set coloring game as the cardinality of
its carrier. It is then a natural question to ask about the size of
the set coloring games constructed in Theorem~\ref{thm:main}. The
following remarks discuss the size of set coloring implementations of
the various gadgets, and how they can be improved. We start with the
primitive gadgets. For brevity, we say that a game $G$ is {\em
  realizable with size} $p$ to mean that it is realizable as a
monotone set coloring game of size $p$.

\begin{proposition}\label{prop:gadgets-quant}
  If $G$ and $H$ are realizable with size $p$ and $q$, respectively,
  then
  \begin{enumerate}\alphalabels
  \item $\g{\top|G}$ and $\g{G|\bot}$ are each realizable with size
    $p+1$,
    \label{prop:gadgets-quant-a}
  \item $\g{\g{\top|G},\g{\top|H}|\g{G|\bot},\g{H|\bot}}$ is
    realizable with size $\max\s{p,q}+2$, and
    \label{prop:gadgets-quant-b}
  \item $\g{G,\g{\top|H}|\g{G|\bot},H}$ is realizable with size
    $p+q+5$. 
    \label{prop:gadgets-quant-c}
  \end{enumerate}
\end{proposition}

\begin{proof}
  Parts {\eqref{prop:gadgets-quant-a}} and
  {\eqref{prop:gadgets-quant-c}} require no special proof, as the
  claimed sizes are just the sizes of the games constructed in the
  proof of Theorem~\ref{thm:main}. For example, $\g{\top|G}$ can be
  realized as $\g{\top|a} + G$ by Lemma~\ref{lem:gadgets}, where
  $\g{\top|a}$ has a carrier of size $1$ by the proof of
  Lemma~\ref{lem:realizable-values}, $G$ has a carrier of size $p$,
  and the sum has a carrier of size $p+1$ by the proof of
  Lemma~\ref{lem:sum-realizable}.  However, there is an important
  optimization in part {\eqref{prop:gadgets-quant-b}}: it turns out
  that in the game
  $\g{\g{\top|a},\g{\top|b}|\g{a|\bot},\g{b|\bot}}+G+H$, it is not
  actually necessary to take the disjoint union of the carriers of $G$
  and $H$; instead, the ordinary union suffices. Informally, this is
  because neither player has an incentive to play in $G$ or $H$ before
  the game $X = \g{\g{\top|a},\g{\top|b}|\g{a|\bot},\g{b|\bot}}$ has
  reached an atomic position, at which point at most one of $G$ or $H$
  needs to be played. More formally, if the carriers of $G$ and $H$
  overlap, the game $G+H$ potentially loses some left options of the
  form $G^L+H$ and $G+H^L$ (which does not matter since they were left
  gift horses anyway), and instead gains new left options of the form
  $G^L+H^L$ (namely, when the left player plays in the common carrier
  of $G$ and $H$). It is easy to check that in the proof of
  Lemma~\ref{lem:gadgets}{\eqref{lem:gadgets-b}}, $X+G^L+H^L$ is still
  a left gift horse for $K$, and therefore any additional left options
  caused by overlapping carriers do not change the result. On the
  other hand, in Proposition~\ref{lem:gadgets}{\eqref{lem:gadgets-c}},
  this is not the case, i.e., when $X=\g{a,\g{\top|b}|\g{a|\bot},b}$,
  then $X+G^L+H^L$ is not in general a left gift horse for $X+G+H$.
\end{proof}

Next, we consider the size of the one-sided choice gadgets.

\begin{proposition}\label{prop:one-sided-choice-quant}
  If $G_1,\ldots,G_n$ are realizable with size $p_1,\ldots,p_n$,
  respectively, then $\g{G_1,\ldots,G_n|\bot}$ is realizable with size
  $\max\s{p_1,\ldots,p_n} + 2\ceil{\log_2 n} + 1$.
\end{proposition}

\begin{proof}
  We first claim that if both $G = \g{G_1,\ldots,G_k|\bot}$ and
  $H = \g{H_1,\ldots,H_l|\bot}$ are realizable with size $p$, then
  $\g{G_1,\ldots,G_k,H_1,\ldots,H_l|\bot}$ is realizable with size
  $p+2$. Indeed, let
  $K=\g{\g{\top|G},\g{\top|H}|\g{G|\bot},\g{H|\bot}}$. Then $K$ is
  realizable with size $p+2$ by
  Proposition~\ref{prop:gadgets-quant}{\eqref{prop:gadgets-quant-b}}. Also,
  by expanding reversible options, it is easily seen that
  $K\eq \g{G_1,\ldots,G_k,H_1,\ldots,H_l|\bot}$.

  The proposition then follows by induction. The base case $n=1$ holds
  by
  Proposition~\ref{prop:gadgets-quant}{\eqref{prop:gadgets-quant-a}}.
  The induction step is an application of the first claim above,
  halving (or nearly halving, in case $n$ is odd) the number of
  options in each step.
\end{proof}

\begin{remark}\label{rem:digits}
  The realization of the one-sided choice gadget given in
  Proposition~\ref{prop:one-sided-choice-quant} admits the following
  concrete description. Let $k=\ceil{\log_2 n}$. The gadget's carrier
  consists of the union of the carriers of $G_1,\ldots,G_n$, together
  with $2k+1$ additional cells $a_1,b_1,\ldots,a_k,b_k,c$.  Each pair
  of cells $(a_i,b_i)$ implements a choice gadget, and we can think of
  these pairs as the ``digits'' of a $k$-bit binary number. The player
  whose turn it is starts filling in the digits from left to right,
  and the other player has no choice but to respond in the same digit
  (or they will lose the game immediately).  In this way, the player
  chooses one of up to $2^k$ possibilities. After all the digits are
  chosen, the player whose turn it is plays in cell $c$. If that
  player was Left, the game continues in whichever of the components
  $G_1,\ldots,G_n$ was chosen. If that player was Right, they win
  immediately.
\end{remark}

We now consider the size of the game $\g{G|H}$ in the coupling lemma
(Lemma~\ref{lem:coupling}).  It turns out that there is a useful
optimization in case the game is of the form
$\g{\g{\top|G}|\g{H|\bot}}$, so we consider this as a special case.

\begin{proposition}\label{prop:coupling-quant}
  Suppose $G$ and $H$ are realizable with size $p$ and $q$,
  respectively.
  \begin{enumerate}\alphalabels
  \item \label{prop:coupling-quant-a} If $\g{G|H}$ is locally
    monotone, then it is realizable with size $p+q+5$.
  \item \label{prop:coupling-quant-b} If $\g{\g{\top|G}|\g{H|\bot}}$
    is locally monotone, then it is realizable with size $p+q+5$.
  \end{enumerate}
\end{proposition}

\begin{proof}
  {\eqref{prop:coupling-quant-a}} As shown in the proof of
  Lemma~\ref{lem:coupling}, we have
  $\g{G|H}\eq\g{G,\g{\top|H}|\g{G|\bot},H}$, so the claim follows by
  Proposition~\ref{prop:gadgets-quant}{\eqref{prop:gadgets-quant-c}}.
  \enskip {\eqref{prop:coupling-quant-b}} Let
  $K=\g{\g{\top|G}|\g{H|\bot}}$. By definition of local monotonicity,
  we have $\g{H|\bot} \leq K \leq \g{\top|G}$.  By definition of the
  order, this implies $H\tri K$ and $K\tri G$, so that $H$ and $G$ are
  left and right gift horses for $K$, respectively. Therefore $K$ is
  equivalent to $\g{H,\g{\top|G}|\g{H|\bot},G}$, and the result
  follows again by
  Proposition~\ref{prop:gadgets-quant}{\eqref{prop:gadgets-quant-c}}.
\end{proof}

We can now consider the size of the game in
Lemma~\ref{lem:nary-coupling}.

\begin{proposition}\label{prop:nary-coupling-quant}
  Consider a game $K=\g{G_1,\ldots,G_n|H_1,\ldots,H_m}$, where
  $G_1,\ldots,G_n,H_1,\ldots,H_m$ are realizable with size
  $p_1,\ldots,p_n,q_1,\ldots,q_m$, respectively.
  \begin{enumerate}\alphalabels
  \item \label{prop:nary-coupling-quant-a} If $K$ is locally
    semi-monotone, then $K$ is realizable with size
    $\max\s{p_1,\ldots,p_n}+\max\s{q_1,\ldots,q_m} + 2\ceil{\log_2 n}
    + 2\ceil{\log_2 m} + 7$.
  \item \label{prop:nary-coupling-quant-b} If $K$ is locally passable,
    then $K$ is realizable with size
    $2\max\s{p_1,\ldots,p_n,q_1,\ldots,q_m} + 2\ceil{\log_2 n}
    + 2\ceil{\log_2 m} + 10$.
  \end{enumerate}
\end{proposition}

\begin{proof}
  {\eqref{prop:nary-coupling-quant-a}} As in the proof of
  Lemma~\ref{lem:nary-coupling}, $K$ can be realized as
  $\g{\g{\top|G}|\g{H|\bot}}$, where $G=\g{G_1,\ldots,G_n|\bot}$ and
  $H=\g{\top|H_1,\ldots,H_m}$. By
  Proposition~\ref{prop:one-sided-choice-quant}, $G$ is realizable
  with size $p=\max\s{p_1,\ldots,p_n}+2\ceil{\log_2 n}+1$, and $H$ is
  realizable with size
  $q=\max\s{q_1,\ldots,q_m} + 2\ceil{\log_2 m} + 1$. By
  Proposition~\ref{prop:coupling-quant}{\eqref{prop:nary-coupling-quant-b}},
  $\g{\g{\top|G}|\g{H|\bot}}$ is realizable with size $p+q+5$.  The
  result follows.
  
  {\eqref{prop:nary-coupling-quant-b}} Since $K$ is passable, it has
  at least one good left option or right option. Without loss of
  generality, assume that $K$ has a good left option $G_i$; the other
  case is dual. It is easy to see (as in the proof of
  \cite[Lemma~6.7]{Selinger2021}) that
  $K'=\s{G_1,\ldots,G_n|H_1,\ldots,H_m,\g{G_i|\bot}}$ is locally
  semi-monotone and equivalent to $K$. By part
  {\eqref{prop:nary-coupling-quant-a}}, $K'$, and therefore $K$, is
  realizable of size
  $\max\s{p_1,\ldots,p_n}+\max\s{q_1,\ldots,q_m,p_i+1} + 2\ceil{\log_2
    n} + 2\ceil{\log_2 (m+1)} + 7$. Since
  $\ceil{\log_2(m+1)}\leq \ceil{\log_2 m}+1$, this implies the claim.
\end{proof}

Putting together the above results, we can get an upper bound on the
size of the set coloring realizations of monotone and passable
games. We define the {\em depth} of a finite game $G$ in the obvious
way, i.e., atomic games have depth 0, and the depth of a composite
game is one more than the maximum depth of its options. We define the
{\em branching factor} of $G$ to be the maximum number of left options
or right options of any position occurring in $G$. The following is a
quantitative version of Theorem~\ref{thm:main}. As before, $A$ is some
fixed atom poset with top and bottom elements.

\begin{proposition}\label{prop:main-quant}
  Let $G$ be a finite game over $A$, and let $d$ and $b$ be its depth
  and branching factor, respectively.
  \begin{enumerate}\alphalabels
  \item \label{prop:main-quant-a} If $G$ is monotone, then $G$ is
    realizable with size $(2^d-1)(4\ceil{\log_2 b}+7)$.
  \item \label{prop:main-quant-b} If $G$ is passable, then $G$ is
    realizable with size $(2^d-1)(4\ceil{\log_2 b}+10)$.
  \end{enumerate}
\end{proposition}

\begin{proof}
  We prove {\eqref{prop:main-quant-a}} by induction. Let
  $C = 4\ceil{\log_2 b}+7$. The claim holds for atomic games, since
  they are realizable with size 0. If
  $G=\s{G_1,\ldots,G_n|H_1,\ldots,H_m}$ is composite, then by
  the induction hypothesis, all of its options are realizable with size
  $p=(2^{d-1}-1)C$.  Applying
  Proposition~\ref{prop:nary-coupling-quant}\eqref{prop:nary-coupling-quant-a},
  we get that $G$ is realizable with size
  \[
    2p + 2\ceil{\log_2 n} + 2\ceil{\log_2 m} + 7 ~~\leq~~ 2p + C ~~=~~
    (2^d-1)C,
  \]
  as claimed.  The proof of {\eqref{prop:main-quant-b}} is almost
  identical, except using
  Proposition~\ref{prop:nary-coupling-quant}\eqref{prop:nary-coupling-quant-b}.
\end{proof}

\begin{remark}
  Naively, it would have been possible to obtain a bound on the size
  of the realization of passable games directly from
  Proposition~\ref{prop:main-quant}\eqref{prop:main-quant-a}, using
  the fact that by Theorem~\ref{thm:fundamental}, every passable game
  is equivalent to a monotone one. However, the translation from
  passable games to monotone games significantly increases the depth
  of the game, in the worst case by a factor of about 5. This would
  have resulted in a bound with a much larger base of exponentiation,
  say, proportional to $32^d$ instead of $2^d$. Conceptually, the
  reason we were able to obtain the better bound of
  Proposition~\ref{prop:main-quant}\eqref{prop:main-quant-b} is that
  while the translation to monotone games increases the depth, it does
  not add any additional coupling gadgets, i.e., the additional depth
  entirely comes from options of the form $\g{G_1,\ldots,G_n|\bot}$ or
  $\g{\top|H_1,\ldots,H_m}$. In fact, we could define the {\em
    coupling depth} of a game, analogous to depth, but such that it
  increases only when a game has non-trivial left and right
  options. In this case, the size of the set coloring realization can
  be bounded as a function of the coupling depth and a small
  logarithmic factor. The passage from passable to monotone games does
  not increase the coupling depth.
\end{remark}

\begin{remark}
  By Proposition~\ref{prop:one-sided-choice-quant} or
  Remark~\ref{rem:digits}, there are monotone set coloring games such
  that the number of cells is strictly smaller (even exponentially
  smaller) than the number of options the left player has in the
  canonical form. The smallest example we know is the game
  $G=\g{a_1,\ldots,a_8|\bot}$, where $a_1,\ldots,a_8$ are incomparable
  atoms.  It is realizable with size $7$ by
  Proposition~\ref{prop:one-sided-choice-quant}.
\end{remark}

\section{Conclusion and future work}

In this paper, we proved that all passable games are realizable as
monotone set coloring games. Our proof also sheds new light on the
structure of passable games: they are generated from the atomic games
by four distinct operations: left and right forcing, choice, and
coupling. Thus, the passable games form a kind of algebra with four
operations. It is an interesting question which equations, if any,
these operations satisfy, but we did not pursue this question
here. The method we gave in this paper is general, and could in
principle be used to show that passable games are realizable by other
classes of games, including non-self-dual games such as Shannon
games. This is left for future work.


\appendix
\section{Appendix: Values over $P_4$ that are realizable with up to 5 cells}
\label{app:a}

The following is a complete list of all game values over the poset
$P_4=\s{\top, a, b, \bot}$ that are realizable as monotone set
coloring game with up to 5 cells. We give a concrete realization of
each value.  Values that can be obtained by duality (exchanging $\top$
and $\bot$) or isomorphism (exchanging $a$ and $b$) have been omitted
from the list. We use $\emptyset$ to denote the empty set, and
$\epsilon$ to denote the unique position on $0$ cells.

\paragraph{0 cells:}~

\begin{mymath}
\begin{array}{ll}
  \val{\top} & a: \s{\epsilon},\enspace b: \s{\epsilon} \\
  \val{a} & a: \s{\epsilon},\enspace b: \emptyset \\
  \val{b} & a: \emptyset,\enspace b: \s{\epsilon} \\
  \val{\bot} & a: \emptyset,\enspace b: \emptyset \\
\end{array}
\end{mymath}

\paragraph{1 cell:}~

\begin{mymath}
\begin{array}{ll}
  \val{\g{\top\mid a}} & a: \s{\ciw},\enspace b: \s{\cib} \\
  \val{\g{\top\mid \bot}} & a: \s{\cib},\enspace b: \s{\cib} \\
\end{array}
\end{mymath}

\paragraph{2 cells:}

Values of the form $\g{\top|G}$ and $\g{G|\bot}$, where $G$ is as above, and:

\begin{mymath}
\begin{array}{ll}
  \val{\g{\g{\top\mid a},\g{\top\mid b}\mid \g{a\mid \bot},\g{b\mid \bot}}} & a: \s{\ciw\cib},\enspace b: \s{\cib\ciw} \\
\end{array}
\end{mymath}

\paragraph{3 cells:}

Values of the form $\g{\top|G}$ and $\g{G|\bot}$, where $G$ is as above, and:

\begin{mymath}
\begin{array}{ll}
  \val{\g{a,b\mid\bot}} & a: \s{\ciw\cib\cib,\cib\ciw\cib},\enspace b: \s{\ciw\cib\cib,\cib\cib\ciw} \\    
  \val{\g{\top\mid a,\g{b\mid \bot}}} & a: \s{\ciw\ciw\cib},\enspace b: \s{\ciw\cib\ciw,\cib\ciw\cib} \\
\end{array}
\end{mymath}

\paragraph{4 cells:}

Values of the form $\g{\top|G}$ and $\g{G|\bot}$, where $G$ is as above, and:

\begin{mymath}
\begin{array}{ll}
  \val{\g{a\mid a,\g{b\mid \bot}}} & a: \s{\ciw\ciw\cib\cib,\ciw\cib\ciw\cib,\cib\ciw\cib\ciw},\enspace b: \s{\cib\cib\ciw\ciw} \\ 
  \val{\g{\g{\top\mid a,b}\mid a,b}} & a: \s{\ciw\ciw\cib\cib,\cib\cib\ciw\ciw},\enspace b: \s{\ciw\cib\ciw\cib,\cib\ciw\cib\ciw} \\
  \val{\g{a,\g{\top\mid a,b}\mid \bot}} & a: \s{\ciw\ciw\cib\cib,\ciw\cib\ciw\cib,\cib\cib\cib\ciw},\enspace b: \s{\ciw\cib\cib\ciw,\cib\ciw\cib\cib} \\
  \val{\g{\g{\top\mid a}\mid a,\g{b\mid \bot}}} & a: \s{\ciw\ciw\ciw\cib,\ciw\cib\cib\ciw},\enspace b: \s{\cib\ciw\cib\ciw} \\
  \val{\g{\g{\top\mid a},\g{\top\mid b}\mid \bot}} & a: \s{\ciw\ciw\ciw\cib,\cib\cib\cib\ciw},\enspace b: \s{\ciw\ciw\cib\cib,\ciw\cib\cib\ciw} \\
  \val{\g{a,\g{\top\mid a,\g{b\mid \bot}}\mid \bot}} & a: \s{\ciw\ciw\cib\cib,\ciw\cib\ciw\cib},\enspace b: \s{\ciw\cib\cib\ciw,\cib\ciw\cib\cib} \\
\end{array}
\end{mymath}

\newpage

\paragraph{5 cells:}

Values of the form $\g{\top|G}$ and $\g{G|\bot}$, where $G$ is as above, and:

\begin{mymath}
\begin{array}{ll}
  \val{\g{a,b\mid a}} & a: \s{\ciw\ciw\ciw\ciw\cib,\ciw\ciw\ciw\cib\ciw},\enspace b: \s{\ciw\ciw\cib\ciw\cib,\ciw\cib\ciw\cib\ciw,\cib\cib\cib\ciw\ciw} \\
  \val{\g{\g{\top\mid a}\mid a,b}} & a: \s{\ciw\ciw\ciw\ciw\cib,\ciw\ciw\cib\cib\ciw,\cib\cib\ciw\ciw\ciw},\enspace b: \s{\ciw\cib\ciw\cib\ciw} \\
  \val{\g{a,b,\g{\top\mid a,b}\mid \bot}} & a: \s{\ciw\ciw\ciw\cib\cib,\ciw\ciw\cib\ciw\cib,\ciw\cib\cib\cib\ciw},\enspace b: \s{\ciw\ciw\cib\cib\cib,\ciw\cib\ciw\cib\ciw,\cib\cib\ciw\ciw\ciw} \\
  \val{\g{\g{a,\g{\top\mid b}\mid a}\mid a,b}} & a: \s{\ciw\ciw\ciw\cib\cib,\ciw\ciw\cib\ciw\cib,\ciw\cib\ciw\cib\ciw,\cib\cib\cib\ciw\ciw},\enspace b: \s{\ciw\cib\cib\ciw\ciw,\cib\ciw\ciw\cib\cib,\cib\ciw\cib\cib\ciw} \\
  \val{\g{a,\g{\top\mid b}\mid a,\g{b\mid \bot}}} & a: \s{\ciw\ciw\ciw\cib\cib,\ciw\cib\cib\ciw\ciw,\cib\ciw\cib\ciw\cib},\enspace b: \s{\ciw\ciw\ciw\cib\ciw,\ciw\cib\ciw\ciw\cib,\cib\cib\cib\ciw\ciw} \\
  \val{\g{a,\g{\top\mid b}\mid \g{a\mid \bot},b}} & a: \s{\ciw\ciw\ciw\cib\cib,\cib\ciw\cib\ciw\cib,\cib\cib\ciw\ciw\ciw},\enspace b: \s{\ciw\ciw\cib\ciw\ciw,\ciw\cib\ciw\cib\ciw} \\
  \val{\g{a,\g{b,\g{\top\mid a}\mid b}\mid \bot}} & a: \s{\ciw\ciw\ciw\cib\cib,\ciw\ciw\cib\ciw\cib,\cib\cib\ciw\cib\ciw},\enspace b: \s{\ciw\cib\cib\ciw\ciw,\cib\cib\ciw\ciw\ciw} \\
  \val{\g{a,\g{\top\mid a,b}\mid a,\g{b\mid \bot}}} & a: \s{\ciw\ciw\ciw\cib\cib,\ciw\ciw\cib\ciw\cib,\ciw\cib\ciw\cib\ciw},\enspace b: \s{\ciw\cib\cib\ciw\ciw,\cib\ciw\cib\ciw\cib} \\
  \val{\g{a,\g{\top\mid b}\mid \g{a\mid a,\g{b\mid \bot}}}} & a: \s{\ciw\ciw\ciw\cib\cib,\ciw\ciw\cib\ciw\cib,\ciw\cib\ciw\ciw\cib,\cib\ciw\ciw\cib\ciw,\cib\ciw\cib\ciw\ciw},\enspace b: \s{\ciw\ciw\ciw\cib\cib,\ciw\cib\ciw\cib\ciw,\cib\cib\ciw\ciw\ciw} \\
  \val{\g{\top\mid \g{a\mid a,\g{b\mid \bot}},\g{b\mid \bot}}} & a: \s{\ciw\ciw\ciw\cib\cib,\ciw\ciw\cib\ciw\cib,\ciw\cib\ciw\ciw\cib,\cib\ciw\ciw\cib\ciw},\enspace b: \s{\ciw\ciw\cib\ciw\cib,\ciw\ciw\cib\cib\ciw,\cib\cib\ciw\ciw\ciw} \\
  \val{\g{a,\g{\top\mid b}\mid \g{a\mid \bot},\g{b\mid \bot}}} & a: \s{\ciw\ciw\ciw\cib\cib,\ciw\ciw\cib\ciw\cib,\ciw\cib\ciw\cib\ciw},\enspace b: \s{\ciw\cib\cib\ciw\ciw,\cib\ciw\cib\ciw\ciw} \\
  \val{\g{a,\g{\top\mid b}\mid a,\g{b\mid b,\g{a\mid \bot}}}} & a: \s{\ciw\ciw\ciw\cib\cib,\ciw\cib\cib\ciw\ciw,\cib\ciw\cib\ciw\cib},\enspace b: \s{\ciw\ciw\cib\cib\ciw,\ciw\cib\ciw\ciw\cib,\cib\ciw\ciw\cib\ciw,\cib\cib\ciw\ciw\ciw} \\
  \val{\g{a,b\mid \g{a,\g{\top\mid a,\g{b\mid \bot}}\mid \bot}}} & a: \s{\ciw\ciw\ciw\ciw\cib,\ciw\ciw\cib\cib\ciw,\ciw\cib\ciw\cib\ciw},\enspace b: \s{\ciw\ciw\cib\cib\cib,\ciw\cib\cib\cib\ciw,\cib\ciw\cib\ciw\cib,\cib\cib\ciw\ciw\ciw} \\
  \val{\g{\top\mid a,\g{b,\g{\top\mid b,\g{a\mid \bot}}\mid \bot}}} & a: \s{\ciw\ciw\ciw\ciw\cib,\ciw\ciw\cib\cib\ciw,\cib\cib\ciw\cib\ciw},\enspace b: \s{\ciw\ciw\cib\ciw\cib,\ciw\cib\ciw\cib\ciw,\cib\cib\ciw\ciw\ciw} \\
  \val{\g{\g{\top\mid a,b},\g{\top\mid a,\g{a,b\mid \bot}}\mid \bot}} & a: \s{\ciw\ciw\ciw\cib\cib,\ciw\cib\cib\ciw\ciw,\cib\ciw\cib\ciw\cib},\enspace b: \s{\ciw\ciw\cib\cib\ciw,\cib\ciw\ciw\cib\cib,\cib\ciw\cib\ciw\cib,\cib\cib\ciw\ciw\cib} \\
  \val{\g{\g{a,\g{\top\mid b}\mid a},\g{b,\g{\top\mid a}\mid b}\mid \bot}} & a: \s{\ciw\ciw\ciw\cib\cib,\ciw\ciw\cib\ciw\cib,\cib\cib\ciw\cib\ciw},\enspace b: \s{\ciw\cib\ciw\cib\cib,\cib\ciw\cib\ciw\ciw,\cib\cib\ciw\ciw\ciw} \\
  \val{\g{a,\g{\top\mid b}\mid \g{a\mid \bot},\g{b\mid b,\g{a\mid \bot}}}} & a: \s{\ciw\ciw\ciw\cib\cib,\ciw\cib\cib\ciw\ciw,\cib\ciw\cib\ciw\cib},\enspace b: \s{\ciw\cib\ciw\cib\ciw,\cib\ciw\ciw\ciw\cib,\cib\cib\ciw\ciw\ciw} \\
  \val{\g{a,\g{b,\g{\top\mid a}\mid b}\mid \g{a\mid \bot},\g{b\mid \bot}}} & a: \s{\ciw\ciw\ciw\cib\cib,\ciw\ciw\cib\ciw\cib,\ciw\cib\ciw\cib\ciw},\enspace b: \s{\cib\ciw\cib\ciw\ciw,\cib\cib\ciw\ciw\ciw} \\
  \val{\g{\g{a,\g{\top\mid b}\mid a},\g{\top\mid a,\g{b\mid \bot}}\mid \bot}} & a: \s{\ciw\ciw\ciw\cib\cib,\ciw\ciw\cib\ciw\cib,\ciw\cib\ciw\cib\ciw},\enspace b: \s{\ciw\ciw\cib\cib\cib,\cib\ciw\cib\ciw\ciw,\cib\cib\ciw\cib\ciw} \\
  \val{\g{\g{\top\mid a,\g{a,\g{\top\mid b}\mid \bot}}\mid a,\g{b\mid \bot}}} & a: \s{\ciw\ciw\ciw\cib\cib,\ciw\ciw\cib\ciw\cib,\cib\cib\ciw\ciw\ciw},\enspace b: \s{\ciw\cib\ciw\cib\ciw,\cib\ciw\cib\ciw\ciw} \\
\end{array}
\end{mymath}
\begin{mymath}
\begin{array}{ll}
  \valy{\g{\top\mid \g{a\mid a,\g{b\mid \bot}},\g{a,\g{\top\mid a,b}\mid \bot}}} & a: \s{\ciw\ciw\ciw\cib\cib,\ciw\ciw\cib\ciw\cib,\ciw\ciw\cib\cib\ciw,\ciw\cib\ciw\ciw\cib,\cib\ciw\ciw\cib\ciw,\cib\cib\cib\ciw\ciw},\\& b: \s{\ciw\ciw\cib\ciw\cib,\cib\ciw\cib\cib\ciw,\cib\cib\ciw\ciw\ciw} \\\\[-1.5ex]
  \valy{\g{\g{\top\mid a},\g{\top\mid b}\mid \g{a\mid a,\g{b\mid \bot}},\g{b\mid \bot}}} & a: \s{\ciw\ciw\ciw\cib\cib,\ciw\ciw\cib\ciw\cib,\ciw\cib\ciw\cib\ciw},\\& b: \s{\ciw\cib\cib\ciw\ciw,\cib\ciw\ciw\ciw\cib,\cib\ciw\ciw\cib\ciw} \\\\[-1.5ex]
  \valy{\g{\g{a,\g{\top\mid b}\mid a},\g{\top\mid a,\g{a,\g{\top\mid b}\mid \bot}}\mid \bot}} & a: \s{\ciw\ciw\ciw\cib\cib,\ciw\ciw\cib\ciw\cib,\ciw\cib\ciw\cib\ciw,\cib\cib\cib\ciw\ciw},\\& b: \s{\ciw\ciw\cib\cib\cib,\cib\ciw\cib\ciw\ciw,\cib\cib\ciw\cib\ciw} \\\\[-1.5ex]
  \valy{\g{a,\g{\top\mid b,\g{b,\g{\top\mid a}\mid \bot}}\mid \g{a\mid \bot},\g{b\mid \bot}}} & a: \s{\ciw\ciw\ciw\cib\cib,\ciw\ciw\cib\ciw\cib,\ciw\cib\ciw\cib\ciw},\\& b: \s{\ciw\cib\cib\ciw\ciw,\cib\ciw\cib\ciw\cib,\cib\ciw\cib\cib\ciw} \\\\[-1.5ex]
  \valy{\g{\g{a,\g{\top\mid b}\mid a},\g{\top\mid a,\g{a,\g{\top\mid b}\mid \bot}}\mid a,b}} & a: \s{\ciw\ciw\ciw\cib\cib,\ciw\ciw\cib\ciw\cib,\ciw\cib\ciw\cib\ciw,\cib\ciw\cib\ciw\ciw},\\& b: \s{\ciw\cib\ciw\cib\cib,\cib\ciw\cib\ciw\cib,\cib\cib\ciw\ciw\ciw} \\\\[-1.5ex]
  \valy{\g{\top\mid a,\g{b\mid b,\g{a\mid \bot}},\g{b,\g{\top\mid b,\g{a\mid \bot}}\mid \bot}}} & a: \s{\ciw\ciw\ciw\cib\cib,\ciw\cib\cib\ciw\ciw,\cib\ciw\cib\ciw\cib},\\& b: \s{\ciw\ciw\cib\cib\cib,\ciw\cib\ciw\cib\ciw,\cib\ciw\ciw\ciw\cib,\cib\ciw\ciw\cib\ciw,\cib\ciw\cib\ciw\ciw} \\\\[-1.5ex]
  \valy{\g{\g{\g{\top\mid a}\mid a,\g{b\mid \bot}},\g{\g{\top\mid b}\mid b,\g{a\mid \bot}}\mid \bot}} & a: \s{\ciw\ciw\cib\cib\cib,\ciw\cib\ciw\cib\cib,\cib\ciw\cib\ciw\cib,\cib\cib\ciw\cib\ciw},\\& b: \s{\cib\ciw\cib\cib\ciw,\cib\cib\ciw\ciw\cib,\cib\cib\cib\ciw\ciw} \\\\[-1.5ex]
  \valy{\g{\g{\top\mid a},\g{\top\mid b}\mid \g{a\mid a,\g{b\mid \bot}},\g{b\mid b,\g{a\mid \bot}}}} & a: \s{\ciw\ciw\ciw\cib\cib,\ciw\ciw\cib\ciw\cib,\ciw\cib\ciw\cib\ciw},\\& b: \s{\ciw\cib\cib\ciw\ciw,\cib\ciw\ciw\ciw\cib,\cib\cib\ciw\ciw\ciw} \\\\[-1.5ex]
  \valy{\g{\g{\top\mid a,\g{a,\g{\top\mid b}\mid \bot}},\g{\top\mid b,\g{b,\g{\top\mid a}\mid \bot}}\mid \bot}} & a: \s{\ciw\ciw\ciw\cib\cib,\ciw\cib\cib\ciw\ciw,\cib\ciw\cib\ciw\cib},\\& b: \s{\ciw\ciw\cib\cib\ciw,\ciw\cib\ciw\cib\cib,\cib\cib\ciw\ciw\ciw} \\\\[-1.5ex]
\end{array}
\end{mymath}
\begin{mymath}
\begin{array}{ll}
\valy{\g{\g{\top\mid a,\g{a,\g{\top\mid b}\mid \bot}},\g{\top\mid b,\g{b,\g{\top\mid a}\mid \bot}}\mid \g{a\mid \bot},\g{b\mid \bot}}} & a: \s{\ciw\ciw\ciw\cib\cib,\ciw\cib\cib\ciw\ciw,\cib\ciw\cib\ciw\cib},\\& b: \s{\ciw\ciw\cib\cib\ciw,\ciw\cib\ciw\ciw\cib,\cib\cib\ciw\cib\ciw} \\\\[-1ex]
\end{array}
\end{mymath}

\newpage
\bibliographystyle{abbrv}
\bibliography{gadgets}

\end{document}